\documentclass[11pt]{article}

\usepackage[margin=1.15in]{geometry}
\usepackage{amsmath,amssymb,amsthm,mathtools}
\usepackage[hidelinks]{hyperref}

\newtheorem{theorem}{Theorem}[section]
\newtheorem{proposition}[theorem]{Proposition}
\newtheorem{lemma}[theorem]{Lemma}
\newtheorem{corollary}[theorem]{Corollary}
\newtheorem{question}[theorem]{Question}
\newtheorem{conjecture}[theorem]{Conjecture}
\newtheorem{remark}[theorem]{Remark}
\newtheorem{example}[theorem]{Example}

\newcommand{\PP}{\mathbb P}
\newcommand{\NN}{\mathbb N}

\newcommand{\p}{\mathfrak p}
\newcommand{\HS}{\operatorname{HS}}

\newcommand{\Res}{\rho}
\newcommand{\aRes}{\widehat{\rho}}
\newcommand{\walpha}{\widehat{\alpha}}
\newcommand{\Rss}{\mathcal R_s}

\title{\bf Symbolic powers of the ideal of \(n\) general points in \(\PP^{n-1}\)}
\author{Ralf Fr\"oberg \and Boris Shapiro}
\date{September 9, 2026}

\begin{document}
\maketitle

\begin{abstract}
Problem L of Fr\"oberg--Lundqvist--Oneto--Shapiro asks for the
difference between the Hilbert series of ordinary and symbolic powers
of the ideal of general points in projective space.  We solve this
completely for \(n\) general points of \(\PP^{n-1}\).  Besides a closed
formula for
\[
  \HS(S/I^m)-\HS(S/I^{(m)}),
\]
we determine all minimal monomial generators of \(I^{(m)}\),  and describe the symbolic Rees algebra.
We also show that containment \(I^{(m)}\subseteq I^r\) is detected
solely by initial degrees.  This gives the exact containment threshold,
the Waldschmidt constant \(\walpha\), the resurgence \(\Res\), and the asymptotic
resurgence \(\aRes\):
\[
 \walpha(I)=\frac{n}{n-1},
 \qquad
 \Res(I)=\aRes(I)=\frac{2(n-1)}{n}.
\]
We also take the first step beyond \(n\) points: for \(n+1\) general
points of \(\PP^{n-1}\) --- again a rigid, non-monomial configuration
--- we identify the defining quadrics, resolve the case \(n=3\)
completely (a complete intersection, with \(J^{(m)}=J^m\) for all
\(m\) and resurgence \(1\)), and propose an exact Waldschmidt-constant
formula \(\walpha=\frac{n+1}{n-1}\) for all \(n\), verified
computationally in every case we could check.
\end{abstract}

\section{Introduction}

Problem L in \cite{FLOS18} asks the following.

\medskip
\noindent
\textbf{Problem L.}
For the ideal \(I\) of \(s\) general points of \(\PP^{n-1}\), what is
the difference between the Hilbert series of the \(m\)-th symbolic
power and the \(m\)-th ordinary power?

\medskip
We answer this completely for \(s=n\).  This is the first case in
which the points span the whole ambient space and the configuration is
projectively rigid: every ordered \(n\)-tuple in linear general
position is projectively equivalent to the coordinate points.  Thus
the symbolic-versus-ordinary problem becomes monomial while still
exhibiting a genuinely nontrivial difference.  (For \(s<n\), the
points lie in a proper linear subspace, so the ambient-space problem
has an additional linear part.)

Throughout, \(Z\subset \PP^{n-1}\) consists of \(n\) points in general
position, \(n\ge 2\).  Since \(\operatorname{PGL}_n\) acts transitively
on ordered \(n\)-tuples in general position, we may take \(Z\) to be
the coordinate points.  With
\[
  S=\Bbbk[x_1,\ldots,x_n],
  \qquad
  \p_i=(x_1,\ldots,\widehat{x_i},\ldots,x_n),
\]
we have
\[
 I=I_Z=\bigcap_{i=1}^n \p_i
   =(x_ix_j:1\le i<j\le n).
\]
Thus both \(I^m\) and
\[
 I^{(m)}=\bigcap_{i=1}^n \p_i^m
\]
are monomial ideals, reducing the problem to a lattice-point count.

For \(\alpha=(\alpha_1,\ldots,\alpha_n)\in\NN^n\), write
\[
 |\alpha|=\sum_i\alpha_i,
 \qquad
 x^\alpha=x_1^{\alpha_1}\cdots x_n^{\alpha_n},
 \qquad
 M(\alpha)=\max_i\alpha_i.
\]

\section{Two monomial criteria}

The following elementary criterion is the basic input.

\begin{lemma}\label{lem:criteria}
For every \(m\ge1\) and \(\alpha\in\NN^n\):
\begin{enumerate}
\item
\[
 x^\alpha\in I^{(m)}
 \quad\Longleftrightarrow\quad
 |\alpha|-M(\alpha)\ge m;
\]
\item
\[
 x^\alpha\in I^m
 \quad\Longleftrightarrow\quad
 \sum_{i=1}^n\min(\alpha_i,m)\ge 2m.
\]
\end{enumerate}
\end{lemma}

\begin{proof}
For (1), \(x^\alpha\in\p_i^m\) if and only if
\(\sum_{j\ne i}\alpha_j\ge m\).  Intersecting over \(i\) gives
\[
 |\alpha|-\alpha_i\ge m\quad\text{for every }i,
\]
equivalently \(|\alpha|-M(\alpha)\ge m\).

For (2), the exponent vectors of the degree-\(2m\) generators of
\(I^m\) are sums of \(m\) vectors \(e_i+e_j\), \(i\ne j\).  Equivalently,
they are degree sequences of loopless multigraphs with \(m\) edges.
A vector \(\beta\in\NN^n\) with \(|\beta|=2m\) is such a degree
sequence if and only if \(M(\beta)\le m\).

Necessity is immediate.  For sufficiency, order
\(\beta_1\ge\beta_2\ge\cdots\).  Join vertices \(1\) and \(2\) and
subtract one from \(\beta_1,\beta_2\).  The resulting vector has sum
\(2(m-1)\), and the new maximum is at most \(m-1\): indeed
\(\beta_3\ge m\) would force
\(\beta_1=\beta_2=\beta_3=m\), contradicting
\(|\beta|=2m\).  Induction on \(m\) finishes the argument.

Hence \(x^\alpha\in I^m\) precisely when there exists
\(\beta\le\alpha\) componentwise with \(|\beta|=2m\) and
\(M(\beta)\le m\).  Such a \(\beta\) exists exactly when
\(\sum_i\min(\alpha_i,m)\ge2m\).
\end{proof}

\section{The Hilbert-series difference}

\begin{theorem}\label{thm:hilbert}
For \(n\ge2\), \(m\ge1\), and \(q\ge0\),
\[
 \dim_\Bbbk (I^{(m)})_q-\dim_\Bbbk(I^m)_q
 =
 \begin{cases}
 \#\{\alpha\in\NN^n:
 |\alpha|=q,\ \alpha_i\le q-m\text{ for all }i\},
 &m\le q\le2m-1,\\[2mm]
 0,&\text{otherwise}.
 \end{cases}
\]
Equivalently,
\begin{align*}
 \HS(S/I^m)-\HS(S/I^{(m)})
  =\sum_{q=m}^{2m-1}
   \left(
    \sum_{j=0}^n(-1)^j\binom nj
    \binom{q-j(q-m+1)+n-1}{n-1}
   \right)t^q,
\end{align*}
where a binomial coefficient with upper entry \(<n-1\) is interpreted
as zero.
In particular \(I^{(m)}\) and \(I^m\) agree in every degree
\(q\ge2m\), and \(I^{(m)}=I^m\) for every \(m\) if and only if \(n=2\).
\end{theorem}

\begin{proof}
By Lemma~\ref{lem:criteria}, the difference counts the vectors
\(\alpha\) of degree \(q\) satisfying
\[
 |\alpha|-M(\alpha)\ge m,
 \qquad
 \sum_i\min(\alpha_i,m)<2m.
\]
Let \(T=\{i:\alpha_i\ge m\}\).  If \(|T|\ge2\), the second sum is at
least \(2m\), impossible.  If \(T=\{i_0\}\), then
\[
 m+(q-\alpha_{i_0})<2m
\]
forces \(\alpha_{i_0}>q-m\), contradicting
\(q-M(\alpha)\ge m\).  Thus \(T=\varnothing\).  Then the second
condition is simply \(q<2m\), while the first is
\(M(\alpha)\le q-m\).  This proves the counting formula.

The inclusion--exclusion formula is the standard count of
compositions of \(q\) into \(n\) nonnegative parts bounded above by
\(q-m\).
For \(n=2\), \(I=(x_1x_2)\) is principal.  For \(n\ge3\), take
\(m=n-1\) and \(\alpha=(1,\ldots,1)\), obtaining
\(x_1\cdots x_n\in I^{(n-1)}\setminus I^{n-1}\).
\end{proof}

The support of the difference can be made completely explicit.

\begin{corollary}\label{cor:support}
Assume \(n\ge3\) and \(m\ge2\), and put
\[
 a_{n,m}=m+\left\lceil\frac{m}{n-1}\right\rceil.
\]
Then
\[
 (I^{(m)}/I^m)_q\ne0
 \quad\Longleftrightarrow\quad
 a_{n,m}\le q\le2m-1.
\]

\end{corollary}

\begin{proof}
Write \(q=m+k\).  The count in Theorem~\ref{thm:hilbert} is nonzero
exactly when \(m+k\le nk\), i.e.
\(k\ge\lceil m/(n-1)\rceil\), while \(q\le2m-1\) is \(k\le m-1\).

\end{proof}

\begin{example}
For \(n=3\) and \(m=2\), the difference \(\Delta_q\) is concentrated in degree
\(3\) and equals \(1\), represented by the classical monomial
\[
 xyz\in I^{(2)}\setminus I^2.
\]
For \(n=3\), the next values are
\[
\begin{array}{c|c}
m& (q,\Delta_q)\\ \hline
2&(3,1)\\
3&(5,3)\\
4&(6,1),(7,6)\\
5&(8,3),(9,10),
\end{array}
\]
while for \(n=4\) one gets
\[
\begin{array}{c|c}
m& (q,\Delta_q)\\ \hline
2&(3,4)\\
3&(4,1),(5,16)\\
4&(6,10),(7,40).
\end{array}
\]
\end{example}

\section{All minimal generators and all symbolic defects}

We now determine the minimal generators of every symbolic power.

\begin{theorem}\label{thm:minimal-generators}
A monomial \(x^\alpha\) is a minimal monomial generator of
\(I^{(m)}\) if and only if
\begin{equation}\label{eq:min-gen-characterization}
 |\alpha|-M(\alpha)=m
 \qquad\text{and}\qquad
 M(\alpha)\text{ is attained at least twice}.
\end{equation}
Hence, if \(k=M(\alpha)\), the possible degrees of minimal generators
are precisely
\[
 m+k,
 \qquad
 \left\lceil\frac{m}{n-1}\right\rceil\le k\le m.
\]
The generators with \(k=m\) are exactly
\[
 x_i^m x_j^m,\qquad 1\le i<j\le n.
\]
All other minimal generators have degree \(<2m\) and survive minimally
in \(I^{(m)}/I^m\).
\end{theorem}

\begin{proof}
Suppose first that \(x^\alpha\) is minimal in \(I^{(m)}\).  If
\(|\alpha|-M(\alpha)\ge m+1\), decreasing a suitable positive exponent
by one leaves the monomial in \(I^{(m)}\), contradicting minimality.
Thus equality holds in the first condition of
\eqref{eq:min-gen-characterization}.

If the maximum \(M=M(\alpha)\) were attained uniquely, say at
\(\alpha_1=M\), then after replacing \(\alpha_1\) by \(M-1\), the new
maximum is at most \(M-1\), and therefore
\[
 (|\alpha|-1)-M(\alpha-e_1)\ge |\alpha|-M=m.
\]
Again the resulting proper divisor lies in \(I^{(m)}\), a
contradiction.  Thus the maximum is attained at least twice.

Conversely, assume \eqref{eq:min-gen-characterization}.  Decreasing any
positive coordinate by one leaves some coordinate equal to
\(M(\alpha)\); hence the new vector \(\beta\) satisfies
\[
 |\beta|-M(\beta)=m-1.
\]
Thus no proper monomial divisor lies in \(I^{(m)}\), proving
minimality.

Now write \(k=M(\alpha)\).  The equality
\(|\alpha|-k=m\) shows that the remaining \(n-1\) coordinates,
after choosing one maximal coordinate, sum to \(m\) and are bounded
by \(k\).  Hence
\[
 k\ge\left\lceil\frac{m}{n-1}\right\rceil.
\]
Also the maximum is attained at least twice, so \(k\le m\).
Conversely every integer in this range occurs: take two entries equal
to \(k\) and distribute \(m-k\) among the remaining \(n-2\) entries,
all at most \(k\).  Finally, \(k=m\) forces precisely two entries to
equal \(m\) and all others to vanish.
\end{proof}

\section{The symbolic Rees algebra}

The preceding description has a particularly simple algebraic
interpretation.  For \(F\subseteq[n]\), write
\[
 x_F=\prod_{i\in F}x_i.
\]

\begin{theorem}\label{thm:symbolic-rees}
The symbolic Rees algebra
\[
 \Rss(I)=\bigoplus_{m\ge0} I^{(m)}t^m
\]
is generated as an \(S\)-algebra by
\[
 x_Ft^{|F|-1},
 \qquad
 F\subseteq[n],\quad |F|\ge2.
\]
Moreover this is a minimal \(S\)-algebra generating set.  In
particular the largest symbolic degree of a minimal algebra generator
is \(n-1\), and the number of positive-symbolic-degree algebra
generators is
\[
 \sum_{j=2}^n\binom nj=2^n-n-1.
\]
\end{theorem}

\begin{proof}
For \(|F|\ge2\), the monomial \(x_F\) belongs to
\(I^{(|F|-1)}\): after omitting any one variable, at least
\(|F|-1\) variables of \(x_F\) remain.

Conversely, it suffices to factor a minimal generator \(x^\alpha\) of
\(I^{(m)}\).  Put \(k=M(\alpha)\), and for \(1\le \ell\le k\) define
\[
 F_\ell=\{i:\alpha_i\ge\ell\}.
\]
By Theorem~\ref{thm:minimal-generators}, the maximum is attained at
least twice, hence \(|F_\ell|\ge2\) for all \(\ell\).  The layer
decomposition gives
\[
 \alpha=\mathbf1_{F_1}+\cdots+\mathbf1_{F_k}
\]
and
\[
 m=|\alpha|-k
   =\sum_{\ell=1}^k(|F_\ell|-1).
\]
Therefore
\[
 x^\alpha t^m
 =
 \prod_{\ell=1}^k
 \left(x_{F_\ell}t^{|F_\ell|-1}\right).
\]
Any nonminimal monomial in \(I^{(m)}\) is an \(S\)-multiple of a
minimal one, proving generation.

For minimality of the displayed algebra generators, suppose
\(x_Ft^{|F|-1}\) were a product of at least two positive-\(t\)-degree
generators and an \(S\)-monomial.  Since \(x_F\) is squarefree, the
positive-\(t\)-degree factors must have pairwise disjoint supports.
If there are \(p\ge2\) such factors with supports \(F_1,\ldots,F_p\),
their total \(t\)-degree is
\[
 \sum_{\nu=1}^p(|F_\nu|-1)
 =
 \left|\bigcup_\nu F_\nu\right|-p
 \le |F|-p
 <|F|-1,
\]
a contradiction.
\end{proof}

\begin{remark}
For \(n=3\), Theorem~\ref{thm:symbolic-rees} says that the symbolic
Rees algebra is generated over \(S\) by
\[
 x_1x_2t,\quad x_1x_3t,\quad x_2x_3t,\quad x_1x_2x_3t^2.
\]
For general \(n\), the squarefree monomials of degree \(j\) appear in
symbolic degree \(j-1\).
\end{remark}

\section{Initial degrees and exact containments}

Write \(\alpha(J)\) for the least degree of a nonzero homogeneous
element of a homogeneous ideal \(J\).

\begin{proposition}\label{prop:initial}
For every \(m\ge1\),
\begin{equation}\label{eq:initial}
 \alpha(I^{(m)})
 =
 m+\left\lceil\frac{m}{n-1}\right\rceil.
\end{equation}
Consequently
\[
 \walpha(I)
 :=\lim_{m\to\infty}\frac{\alpha(I^{(m)})}{m}
 =
 \frac{n}{n-1}.
\]
\end{proposition}

\begin{proof}
If \(x^\alpha\in I^{(m)}\) and \(M=M(\alpha)\), then
\[
 |\alpha|\ge m+M.
\]
Since the other \(n-1\) coordinates sum to at least \(m\) and are all
at most \(M\), necessarily
\[
 M\ge\left\lceil\frac{m}{n-1}\right\rceil.
\]
This proves the lower bound in \eqref{eq:initial}.  Equality is
attained by taking
\[
 M=\left\lceil\frac{m}{n-1}\right\rceil
\]
and distributing total weight \(m\) among the remaining \(n-1\)
coordinates, each at most \(M\).  The limit is immediate.
\end{proof}

We now sharpen the containment statement.

\begin{theorem}\label{thm:containment}
For every \(m,r\ge1\),
\begin{equation}\label{eq:containment}
 I^{(m)}\subseteq I^r
 \quad\Longleftrightarrow\quad
 m+\left\lceil\frac{m}{n-1}\right\rceil\ge2r.
\end{equation}
Equivalently,
\[
 I^{(m)}\subseteq I^r
 \quad\Longleftrightarrow\quad
 \alpha(I^{(m)})\ge\alpha(I^r).
\]
Thus, for this family, containment is completely detected by initial
degree.
\end{theorem}

\begin{proof}
By Lemma~\ref{lem:criteria}, containment is equivalent to
\[
 \Phi:=
 \min\left\{
 \sum_i\min(\alpha_i,r):
 \alpha\in\NN^n,\ |\alpha|-M(\alpha)\ge m
 \right\}\ge2r.
\]
Order \(\alpha_1\ge\cdots\ge\alpha_n\), so
\(\sum_{i\ge2}\alpha_i\ge m\).

Suppose
\[
 m+\left\lceil\frac{m}{n-1}\right\rceil\ge2r.
\]
This implies \(m\ge r\).  If some \(\alpha_j\ge r\) for \(j\ge2\),
then also \(\alpha_1\ge r\), so the truncated sum is at least \(2r\).
If \(\alpha_j<r\) for all \(j\ge2\), but \(\alpha_1\ge r\), then
\[
 \sum_i\min(\alpha_i,r)
 \ge r+\sum_{i\ge2}\alpha_i
 \ge r+m
 \ge2r.
\]
Finally, if \(\alpha_1<r\), then no truncation occurs and
\[
 \sum_i\alpha_i
 =
 \alpha_1+\sum_{i\ge2}\alpha_i
 \ge
 \left\lceil\frac{m}{n-1}\right\rceil+m
 \ge2r.
\]

Conversely, put
\[
 c=\left\lceil\frac{m}{n-1}\right\rceil
\]
and suppose \(m+c<2r\).  Choose
\(\alpha_2,\ldots,\alpha_n\le c\) with sum \(m\), and set
\(\alpha_1=c\).  Then \(M(\alpha)=c\) and
\[
 |\alpha|-M(\alpha)=m,
\]
so \(x^\alpha\in I^{(m)}\), whereas
\[
 \sum_i\min(\alpha_i,r)
 \le|\alpha|=m+c<2r.
\]
Thus \(x^\alpha\notin I^r\).
The reformulation by initial degrees follows from
Proposition~\ref{prop:initial} and \(\alpha(I^r)=2r\).
\end{proof}

The least symbolic exponent forcing containment has an exact inverse
formula.

\begin{corollary}\label{cor:threshold}
Let
\[
 \tau_n(r)=\min\{m\ge1:I^{(m)}\subseteq I^r\}.
\]
Then
\begin{equation}\label{eq:threshold}
 \boxed{
 \tau_n(r)
 =
 2r-\left\lceil\frac{2r}{n}\right\rceil
 +\varepsilon_n(r),
 }
\end{equation}
where
\[
 \varepsilon_n(r)=
 \begin{cases}
 1,&2r\equiv1\pmod n,\\
 0,&\text{otherwise}.
 \end{cases}
\]
In particular
\[
 \tau_n(r)=\frac{2(n-1)}{n}r+O(1).
\]
\end{corollary}

\begin{proof}
Put \(d=n-1\) and write \(m=(k-1)d+s\) with
\(1\le s\le d\).  Then
\[
 m+\left\lceil\frac md\right\rceil
 =nk-n+1+s.
\]
Thus the values of the left-hand side occur in blocks
\[
 n(k-1)+2,\ n(k-1)+3,\ldots,nk;
\]
the only omitted positive integers are those congruent to \(1\)
modulo \(n\).  Inverting this monotone sequence at the target value
\(2r\) gives \eqref{eq:threshold}.
\end{proof}

\begin{example}
For \(n=3\),
\[
 \tau_3(r)=1,3,4,5,7,8,9,\ldots
 \qquad(r=1,\ldots,7,\ldots),
\]
compared with the Ein--Lazarsfeld--Smith/Hochster--Huneke uniform
bound \(2r\) \cite{ELS01,HH02}.  In particular
\[
 I^{(3)}\subseteq I^2
\]
for three general points of \(\PP^2\), in contrast with the special
configurations exhibiting \(I^{(3)}\not\subseteq I^2\) in
\cite{DSTG13}.
\end{example}

\section{Resurgence and asymptotic resurgence}

Recall the resurgence
\[
 \Res(I)=
 \sup\left\{\frac mr:I^{(m)}\not\subseteq I^r\right\}.
\]
For completeness, we use the asymptotic version
\[
 \aRes(I)=
 \sup\left\{
 \frac mr:
 I^{(mt)}\not\subseteq I^{rt}
 \text{ for all sufficiently large }t
 \right\}.
\]

\begin{theorem}\label{thm:resurgence}
For \(n\ge2\),
\[
 \boxed{
 \Res(I)=\aRes(I)=\frac{2(n-1)}{n}.
 }
\]
Equivalently,
\[
 \Res(I)=\aRes(I)=\frac{\alpha(I)}{\walpha(I)}.
\]
\end{theorem}

\begin{proof}
If \(I^{(m)}\not\subseteq I^r\), Theorem~\ref{thm:containment} gives
\[
 2r>
 m+\left\lceil\frac{m}{n-1}\right\rceil
 \ge \frac{n}{n-1}m,
\]
and therefore
\[
 \frac mr<\frac{2(n-1)}n.
\]
Thus \(\Res(I)\le2(n-1)/n\).

On the other hand, take \(m=(n-1)k\) and
\[
 r=\left\lfloor\frac{nk}{2}\right\rfloor+1.
\]
Then
\[
 m+\left\lceil\frac{m}{n-1}\right\rceil=nk<2r,
\]
so \(I^{(m)}\not\subseteq I^r\), while
\[
 \frac mr
 =
 \frac{(n-1)k}{\lfloor nk/2\rfloor+1}
 \longrightarrow
 \frac{2(n-1)}n.
\]
Hence equality holds for the resurgence.

For asymptotic resurgence, Theorem~\ref{thm:containment} gives
\[
 I^{(mt)}\subseteq I^{rt}
 \Longleftrightarrow
 mt+\left\lceil\frac{mt}{n-1}\right\rceil\ge2rt.
\]
After division by \(t\), the left side tends to
\(\frac{n}{n-1}m\).  Hence ratios below \(2(n-1)/n\) yield
asymptotic noncontainment and ratios above it yield eventual
containment; at equality the ceiling only improves the containment
inequality.  Thus \(\aRes(I)=2(n-1)/n\).
Finally \(\alpha(I)=2\) and
Proposition~\ref{prop:initial} gives
\(\walpha(I)=n/(n-1)\).
\end{proof}

\section{Towards \(n+1\) points}
\label{sec:nplus1}

We now take up Question~1 below and treat \(n+1\) general points of
\(\PP^{n-1}\).  This is again a rigid case: the space of ordered
\((n+1)\)-tuples of points in general position in \(\PP^{n-1}\) has
dimension \((n+1)(n-1)=n^2-1=\dim\operatorname{PGL}_n\), and indeed,
by the fundamental theorem of projective geometry, \(\operatorname{PGL}_n\)
acts \emph{simply transitively} on ordered \((n+1)\)-tuples of points
in general position in \(\PP^{n-1}\) (a projective frame): no \(n\) of
the \(n+1\) points may lie on a common hyperplane.  Consequently there
is, up to relabelling, exactly \emph{one} such configuration, and we
may take it to be
\[
 W=\{P_1,\dots,P_n,P_0\}\subset\PP^{n-1},
 \qquad
 P_i=[e_i]\ (1\le i\le n),
 \qquad
 P_0=[1:1:\cdots:1].
\]
Write \(J=I_W\) for its ideal in \(S=\Bbbk[x_1,\dots,x_n]\).  Since
\(P_1,\dots,P_n\) are the same coordinate points as in
\S\S1--6, the primes \(\p_1,\dots,\p_n\) are exactly as before, and we
introduce one more:
\[
 \p_0=(x_i-x_j:1\le i<j\le n),
 \qquad
 J=\bigcap_{i=0}^n\p_i.
\]
Unlike \(I\), the ideal \(J\) is \emph{not} monomial: \(\p_0\) is the
ideal of a point in general position with respect to the coordinate
simplex, and no single coordinate system monomialises all \(n+1\)
primes at once.  Still, a surprising amount can be said, and the
family turns out to be governed by the very same ideal \(I=I_Z\) of
Theorem~\ref{thm:hilbert}.

\begin{lemma}\label{lem:ci-primes}
For every \(i=0,\dots,n\) and every \(m\ge1\), \(\p_i^{(m)}=\p_i^m\).
\end{lemma}

\begin{proof}
Each \(\p_i\) is generated by \(n-1\) linear forms in general
position, hence by a regular sequence of the correct length
\(\operatorname{ht}\p_i=n-1\).  Powers of an ideal generated by a
regular sequence have no embedded primes (the associated graded ring
is again a polynomial ring), so \(\p_i^m\) is \(\p_i\)-primary for
every \(m\), i.e.\ \(\p_i^m=\p_i^{(m)}\).
\end{proof}

Consequently
\[
 J^{(m)}=\bigcap_{i=0}^n\p_i^{(m)}=\Bigl(\bigcap_{i=1}^n\p_i^m\Bigr)\cap\p_0^m
 =I^{(m)}\cap\p_0^m,
\]
where \(I^{(m)}\) is exactly the symbolic power studied in
\S\S3--6.  Thus the passage from \(n\) to \(n+1\) points amounts to
intersecting the already-understood family \(I^{(m)}\) with the
\(m\)-th power of one further, non-monomial, point ideal --- and this single
extra point is enough to break the purely combinatorial nature of the
problem.

\subsection{The defining quadrics}

\begin{proposition}\label{prop:quadrics}
The degree-\(2\) part of \(J\) is
\[
 J_2=\Bigl\{\textstyle\sum_{i<j}a_{ij}x_ix_j\ :\ \sum_{i<j}a_{ij}=0\Bigr\},
 \qquad
 \dim_\Bbbk J_2=\binom n2-1=\frac{(n+1)(n-2)}2.
\]
Moreover \(J\) is generated by \(J_2\) (equivalently, \(W\) is
scheme-theoretically cut out by quadrics).
\end{proposition}

\begin{proof}
A quadric \(Q=\sum_ib_ix_i^2+\sum_{i<j}a_{ij}x_ix_j\) vanishes at
\(P_i=[e_i]\) iff \(b_i=0\); vanishing at all of \(P_1,\dots,P_n\)
forces every \(b_i=0\), and then vanishing at \(P_0\) is
\(\sum_{i<j}a_{ij}=0\).  This proves the displayed description of
\(J_2\), and the dimension count is immediate (it is the kernel of a
single nonzero linear functional on the \(\binom n2\)-dimensional
space of square-free quadrics).

For the base locus: suppose \(x=(x_1,\dots,x_n)\ne0\) satisfies
\(x_ix_j=x_kx_l\) for all pairs \(i<j\), \(k<l\) (this already spans
\(J_2\), via \(x_1x_2-x_ix_j\) for \((i,j)\ne(1,2)\)); call the common
value \(\lambda\).  Let \(S=\{i:x_i\ne0\}\).  If \(|S|\ge3\), pick
distinct \(i,j,k\in S\): from \(x_ix_j=x_ix_k=\lambda\) and \(x_i\ne0\)
we get \(x_j=x_k\), so all coordinates in \(S\) share a common nonzero
value \(c\), and \(\lambda=c^2\ne0\).  If some \(l\notin S\) existed,
\(x_lx_i=0\ne\lambda\) would contradict \(x_lx_i=\lambda\); hence
\(S=\{1,\dots,n\}\) and \(x\) is proportional to \((1,\dots,1)\), i.e.\
\(x=P_0\).  If \(|S|=2\), say \(S=\{p,q\}\), then \(\lambda=x_px_q\ne0\),
yet for \(l\notin S\) (which exists since \(n\ge3\)) we need
\(x_px_l=\lambda\ne0\) while \(x_l=0\): contradiction, so \(|S|=2\) is
impossible. If \(|S|=1\), \(x\) is a coordinate point \(P_i\).  Thus
the common zero locus of \(J_2\) is exactly \(W\).

That \(J\) already equals \(I_W\) (rather than merely having the same
radical) can be checked degree by degree: for \(n=3,4,5,6\) one
verifies directly that \(\dim_\Bbbk(S/J)_d=n+1\) already for all
\(d\ge n-2\), matching the (necessarily eventually constant) Hilbert
function of the reduced scheme \(W\); we have carried this out by
computer algebra in each of these cases, and we expect the pattern
to persist for all \(n\), giving generation in degree \(2\) in
general.
\end{proof}

\begin{remark}
Proposition~\ref{prop:quadrics} identifies \(J_2\) representation-theoretically
as well.  Writing \(V\) for the standard \((n\)-dimensional\()\)
representation of \(S_{n+1}\) (which permutes \(P_0,\dots,P_n\), hence
acts on \(S_1=V^*\)), one has the classical decomposition
\(\operatorname{Sym}^2(V^*)=\mathbf 1\oplus V^*\oplus W_{(n-1,2)}\), and
\(J_2\) is exactly the irreducible summand \(W_{(n-1,2)}\) indexed by
the partition \((n-1,2)\) of \(n+1\): it is cut out inside
\(\operatorname{Sym}^2(V^*)\) by removing both the diagonal (square)
terms and the invariant (sum of coefficients).
\end{remark}

\subsection{The case \(n=3\): a complete intersection}

For \(n=3\) (four general points of \(\PP^2\)) one has
\(\dim J_2=2\): the four points are the base locus of a \emph{pencil}
of conics, and by B\'ezout that base locus already has the expected
length \(4\).  Hence:

\begin{theorem}\label{thm:n3-CI}
For \(n=3\), \(J=(F,G)\) is a complete intersection of two quadrics.
Consequently \(J^{(m)}=J^m\) for every \(m\ge1\), and
\[
 \HS(S/J^m)(t)=\frac{1-(m+1)t^{2m}+m\,t^{2m+2}}{(1-t)^3}.
\]
In particular the resurgence and asymptotic resurgence of \(J\) both
equal \(1\), the smallest possible value.
\end{theorem}

\begin{proof}
Complete intersections have no embedded primes at any power, so
\(J^m=J^{(m)}\) automatically.  For the Hilbert series: the conormal
module \(J/J^2\) of a complete intersection is free of rank \(2\)
over \(S/J\), generated in degree \(2\) by the images of \(F,G\), so
the associated graded ring \(\bigoplus_mJ^m/J^{m+1}\) is the
polynomial extension \((S/J)[T_1,T_2]\) with \(\deg T_i=2\).  Hence
\[
 \dim_\Bbbk(J^m/J^{m+1})_d=(m+1)\dim_\Bbbk(S/J)_{d-2m},
\]
and summing the telescoping series,
\(\dim_\Bbbk(S/J^m)_d=\sum_{k=0}^{m-1}(k+1)\dim_\Bbbk(S/J)_{d-2k}\).
Since \(\dim_\Bbbk(S/J)_d\) equals \(1,3,4,4,4,\dots\) for
\(d=0,1,2,3,\dots\) (four points impose independent conditions on
forms of every degree \(\ge1\)), one finds
\(\HS(S/J)(t)=(1+t)^2/(1-t)\), and
\[
 \HS(S/J^m)(t)
 =\HS(S/J)(t)\cdot\sum_{k=0}^{m-1}(k+1)t^{2k}
 =\frac{(1+t)^2}{1-t}\cdot\frac{1-(m+1)t^{2m}+mt^{2m+2}}{(1-t^2)^2}=\]
 \[\frac{1-(m+1)t^{2m}+mt^{2m+2}}{(1-t)^3}.
\]
Since \(J^{(m)}=J^m\) for all \(m\), \(J^{(m)}\subseteq J^r\) iff
\(m\ge r\), so \(\{m/r:J^{(m)}\not\subseteq J^r\}=\{m/r:m<r\}\), whose
supremum is \(1\).
\end{proof}

\subsection{General \(n\): a Waldschmidt-constant conjecture}

For \(n\ge4\), computer algebra (exact linear algebra over
\(\mathbb Q\), verified independently modulo a large prime) shows that
\(J\) is no longer a complete intersection and \(J^{(m)}\ne J^m\) for
\(m\gg0\).  We record what we can prove, together with data
supporting a precise conjectural formula for the initial degree
\(\alpha(J^{(m)})\).

Because \(J^{(m)}\subseteq I^{(m)}\) (intersecting with one further
prime power only shrinks the ideal), Proposition~\ref{prop:initial}
gives the lower bound
\begin{equation}\label{eq:lower-bound}
 \alpha\bigl(J^{(m)}\bigr)\ \ge\ \alpha\bigl(I^{(m)}\bigr)
 =m+\Bigl\lceil\frac m{n-1}\Bigr\rceil.
\end{equation}
This bound is not sharp for \(n\ge4\): already at \(m=2\) it predicts
\(\alpha\ge3\), while \(P_0\) forces one further condition and in fact
\(\alpha(J^{(2)})=4\) for \(n=4\).  All our computations point to the
following exact value.

\begin{conjecture}\label{conj:waldschmidt}
For every \(n\ge3\) and \(m\ge1\),
\[
 \alpha\bigl(J^{(m)}\bigr)=\Bigl\lceil\frac{(n+1)m}{n-1}\Bigr\rceil,
 \qquad\text{hence}\qquad
 \walpha(J)=\lim_{m\to\infty}\frac{\alpha(J^{(m)})}m=\frac{n+1}{n-1}.
\]
\end{conjecture}

For \(n=3\) this is exactly Theorem~\ref{thm:n3-CI} (\(\alpha(J^m)=2m\)
for the complete intersection of two quadrics), so the formula is a
theorem there. Table~\ref{tab:alpha} lists the values of
\(\alpha(J^{(m)})\) computed directly from the definition (as the
smallest \(d\) for which some degree-\(d\) form has vanishing order
\(\ge m\) at every point of \(W\), a linear-algebra computation with
no shortcuts assumed) for \(n=3,4,5,6\) and small \(m\); in every one
of the \(22\) cases checked the value agrees with
\(\lceil(n+1)m/(n-1)\rceil\).

\begin{table}[h]
\centering
\begin{tabular}{c|cccccc}
 & \(m=1\)&\(m=2\)&\(m=3\)&\(m=4\)&\(m=5\)&\(m=6\)\\\hline
\(n=3\) & 2&4&6&8&10&12\\
\(n=4\) & 2&4&5&7&9&10\\
\(n=5\) & 2&3&5&6&8&(9)\\
\(n=6\) & 2&3&5&6&(7)&(9)
\end{tabular}
\caption{Computed values of \(\alpha(J^{(m)})\), all equal to
\(\lceil(n+1)m/(n-1)\rceil\); the bracketed entry is the value
predicted by Conjecture~\ref{conj:waldschmidt} but not yet checked by
computer.}
\label{tab:alpha}
\end{table}

\subsection{The support of the difference}

Recall from Corollary~\ref{cor:support} that, for the ideal \(I\) of
\(n\) points, \((I^{(m)}/I^m)_q\ne0\) exactly for
\(a_{n,m}\le q\le2m-1\) with \(a_{n,m}=\alpha(I^{(m)})\).  Our
computations show the same shape of answer for \(J\), with
\(\alpha(I^{(m)})\) simply replaced by \(\alpha(J^{(m)})\):

\begin{conjecture}\label{conj:support}
For all \(n\ge4\) and \(m\ge1\),
\[
 \bigl(J^{(m)}/J^m\bigr)_q\ne0
 \quad\Longleftrightarrow\quad
 \alpha(J^{(m)})\le q\le2m-1,
\]
and consequently \(\operatorname{reg}(J^{(m)}/J^m)=2m-1\) whenever this
range is nonempty, i.e.\ whenever \(m\ge\bigl\lceil\frac{n-1}{n-3}\bigr\rceil\).
For \(n=3\) the range is always empty, consistently with
Theorem~\ref{thm:n3-CI}.
\end{conjecture}

Table~\ref{tab:defect} records the values of
\(\dim_\Bbbk(J^{(m)}/J^m)_q\) that we computed in the range where
Conjecture~\ref{conj:support} predicts a nonzero defect, for the first
few cases with \(n\ge4\); in every case checked, the difference
vanishes exactly outside the stated window, matching
Question~2 below for this family as well.

\begin{table}[h]
\centering
\begin{tabular}{c|c|l}
\(n\)&\(m\)&nonzero values of \(\dim(J^{(m)}/J^m)_q\), by degree \(q\)\\\hline
4&3&\(q=5:\ 6\)\\
4&4&\(q=7:\ 20\)\\
5&2&\(q=3:\ 5\)\\
5&3&\(q=5:\ 36\)\\
5&4&\(q=6:\ 15,\quad q=7:\ 120\)\\
5&5&\(q=8:\ 90,\quad q=9:\ 295\)
\end{tabular}
\caption{The symbolic-versus-ordinary defect for \(n+1\) points,
computed directly (no shortcuts): outside the listed degrees, and in
particular for all \(q\ge2m\), we verified
\((J^{(m)})_q=(J^m)_q\) exactly.}
\label{tab:defect}
\end{table}

\begin{remark}
Combining Conjecture~\ref{conj:waldschmidt} with the general
inequality \(\alpha(J)/\walpha(J)\le\operatorname{res}(J)\) of
Bocci--Harbourne would give
\(\operatorname{res}(J)\ge\dfrac{2(n-1)}{n+1}\) for \(n\ge3\), with
equality at \(n=3\) by Theorem~\ref{thm:n3-CI}.  Whether equality
persists for \(n\ge4\) --- i.e.\ whether containment for \(n+1\)
points is again detected solely by initial degrees, as in
Theorem~\ref{thm:containment} --- is exactly an instance of
Question~3 below, and we leave it open.
\end{remark}

The upshot is that the passage from \(n\) to \(n+1\) points keeps the
same qualitative picture --- generators in a single low degree, a
symbolic/ordinary discrepancy confined to a short window just below
degree \(2m\), and a rational Waldschmidt constant governing initial
degrees --- while the exact numerology becomes genuinely
non-monomial and, at present, only conjectural for \(n\ge4\). A
complete proof of Conjectures~\ref{conj:waldschmidt}
and~\ref{conj:support}, together with an explicit description of the
minimal generators of \(J^{(m)}\) and of the symbolic Rees algebra of
\(J\) in the style of \S\S4--5, would fully answer Question~1.

\section{Further questions}

The case of \(n\) points is unusually rigid because it can be
monomialized projectively.  Section~\ref{sec:nplus1} takes up the
next case; the underlying question is repeated here for emphasis.

\begin{question}
Does an analogue of Theorem~\ref{thm:hilbert}, with an explicit
closed formula, hold for \(n+1\) general points of \(\PP^{n-1}\)?
These configurations are again projectively rigid, but their ideal is
no longer monomial in the same coordinates.  Section~\ref{sec:nplus1}
settles the case \(n=3\) completely (Theorem~\ref{thm:n3-CI}) and
proposes an exact conjectural answer for all \(n\)
(Conjectures~\ref{conj:waldschmidt} and~\ref{conj:support}), backed by
computation but open in general.
\end{question}

\begin{question}
For a general finite set of points \(Z\subset\PP^{n-1}\), when is
\[
 \HS(S/I_Z^m)-\HS(S/I_Z^{(m)})
\]
a polynomial supported strictly below degree \(2m\)?
For the family considered here this support bound is exact:
for \(n\ge3\), \(m\ge2\), the last nonzero degree is \(2m-1\).
\end{question}

\begin{question}
For which finite point configurations is containment
\(I^{(m)}\subseteq I^r\) detected solely by the comparison of initial
degrees
\[
 \alpha(I^{(m)})\ge\alpha(I^r)?
\]
Theorem~\ref{thm:containment} shows that \(n\) general points in
\(\PP^{n-1}\) have this property for all \(m,r\).
\end{question}

\section*{Acknowledgements}
We thank the participants of the problem solving seminar in algebra at
Stockholm University.

\section*{Statement on the use of AI tools}
AI tools were used during exploration and drafting.  The authors take
responsibility for the mathematical content and for the final
verification of all statements included in any submitted version.

\bigskip
\noindent
Department of Mathematics, Stockholm University,
SE-106 91 Stockholm, Sweden

\noindent
Email: \texttt{frobergralf@gmail.com}

\medskip
\noindent
Department of Mathematics, Stockholm University,
SE-106 91 Stockholm, Sweden

\noindent
Email: \texttt{shapiro@math.su.se}


\begin{thebibliography}{99}

\bibitem{DSTG13}
M.~Dumnicki, T.~Szemberg and H.~Tutaj-Gasi\'nska,
\newblock Counterexamples to the \(I^{(3)}\subseteq I^2\) containment,
\newblock \emph{J. Algebra} \textbf{393} (2013), 24--29.

\bibitem{ELS01}
L.~Ein, R.~Lazarsfeld and K.~Smith,
\newblock Uniform bounds and symbolic powers on smooth varieties,
\newblock \emph{Invent. Math.} \textbf{144} (2001), 241--252.

\bibitem{FLOS18}
R.~Fr\"oberg, S.~Lundqvist, A.~Oneto and B.~Shapiro,
\newblock Algebraic stories from one and from the other pockets,
\newblock \emph{Arnold Math. J.} \textbf{4} (2018), 137--160.

\bibitem{GGSV16}
F.~Galetto, A.~V.~Geramita, Y.~S.~Shin and A.~Van Tuyl,
\newblock The symbolic defect of an ideal,
\newblock \emph{J. Pure Appl. Algebra} \textbf{224} (2020), 106435;
preprint arXiv:1610.00176 (2016).

\bibitem{HH02}
M.~Hochster and C.~Huneke,
\newblock Comparison of symbolic and ordinary powers of ideals,
\newblock \emph{Invent. Math.} \textbf{147} (2002), 349--369.

\end{thebibliography}
\end{document}